\documentclass[3p]{elsarticle}

\usepackage{hyperref, amssymb, amsmath, amsthm, subfig}
\usepackage[font=scriptsize]{caption}%

\newtheorem{theorem}{Theorem}[section]
\newtheorem{lemma}[theorem]{Lemma}
\newtheorem{proposition}[theorem]{Proposition}
\newtheorem{corollary}[theorem]{Corollary}

\theoremstyle{definition}
\newtheorem{definition}[theorem]{Definition}
\newtheorem{example}[theorem]{Example}

\theoremstyle{remark}
\newtheorem{remark}[theorem]{Remark}

\journal{Discrete Mathematics}

\makeatletter
\def\ps@pprintTitle{%
 \let\@oddhead\@empty
 \let\@evenhead\@empty
 \def\@oddfoot{}%
 \let\@evenfoot\@oddfoot}
\makeatother
\begin{document}

\begin{frontmatter}

\title{Positive definite functions on semilattices}

\author[unsw]{V. Kaarnioja\corref{cor1}}
\ead{vesa.kaarnioja@iki.fi}
\cortext[cor1]{Corresponding author}
\author[tuni]{P. Haukkanen}
\ead{pentti.haukkanen@tuni.fi}
\author[aalto]{P. Ilmonen}
\ead{pauliina.ilmonen@aalto.fi}
\author[tuni]{M. Mattila}
\ead{mika.mattila@tuni.fi}
\address[unsw]{School of Mathematics and Statistics, University of New South Wales, Sydney NSW 2052, Australia}
\address[tuni]{Faculty of Information Technology and Communication Sciences, Tampere University, FI-33014, Finland}
\address[aalto]{Aalto University School of Science, Department of Mathematics and Systems Analysis, P.O. Box 11100, FI-00076, Finland}

\begin{abstract}
We introduce a notion of positive definiteness for functions $f\!:P\to\mathbb{R}$ defined on meet semilattices $(P,\preceq,\wedge)$ and prove several properties for these functions. In addition, we utilize the $LDL^{\rm T}$ decomposition of meet matrices in order to explore the properties of multivariate positive definite arithmetic functions $f\!:\mathbb{Z}_+^d\to\mathbb{R}$. Finally, we give a series of examples and counterexamples of positive definite functions.
\end{abstract}

\begin{keyword}
arithmetic function\sep positive definite function\sep meet matrix\sep GCD matrix\sep semilattice
\MSC[2010] 06A12\sep 11A25\sep 11C20\sep 15A69\sep 15B36
\end{keyword}

\end{frontmatter}

\section{Introduction}

Recently, a notion of positive definite arithmetic functions $f\!:\mathbb{Z}_+\to\mathbb{R}$ was introduced in~\cite{HaukkanenMattila18}. The definition given in~\cite{HaukkanenMattila18} is closely connected to the structure theory~\cite{Ovall} and positive definiteness~\cite{BeLi} of GCD matrices defined on the divisor lattice $(\mathbb{Z}_+,|)$.

In this paper, we explore an extension of the aforementioned notion of positive definiteness for a more general class of functions $f\!:P\to\mathbb{R}$ defined on meet semilattices $(P,\preceq,\wedge)$. We discuss several properties that positive definiteness imposes on three classes of functions:
\begin{itemize}
\item[(A)] functions $f\!:P\to\mathbb{R}$ defined on meet semilattices $(P,\preceq,\wedge)$,
\item[(B)] multivariate functions $f\!:P_1\times\cdots\times P_d\to\mathbb{R}$ defined on the product of meet semilattices $(P_i,\preceq_{P_i},\wedge_{P_i})$, $i\in\{1,\ldots,d\}$,
\item[(C)] multivariate arithmetic functions $f\!:\mathbb{Z}_+^d\to\mathbb{R}$.
\end{itemize}
We remark that class (C) is a special case of class (B), whereas class (B) is a special case of the most general case (A). On the other hand, the results of~\cite{HaukkanenMattila18} can thus be seen as a special case of all of the above classes.

For recent theoretical work on multivariate arithmetic functions, we refer to the excellent treatise by T\'{o}th~\cite{Toth}.

We give the basic notations and definitions related to meet matrices in Subsections~\ref{sec:meetsemilattice} and~\ref{sec:notations}. In Section~\ref{sec:poset} we give the general definition of positive definiteness for functions belonging to class (A) and establish the basic properties induced by positive definiteness on these functions. In Section~\ref{sec:multivariate}, we first prove a series of results relating to the structure of meet matrices with Cartesian product form to support our analysis. In Section~\ref{sec:multivariatearithmetic}, we discuss how the results of the previous sections can be specialized to class (C) of multivariate arithmetic functions. Finally, we present a series of examples of positive definite functions in several variables as well as accompanying counterexamples in Section~\ref{sec:examples}. We end this paper with conclusions on the results.

\subsection{Meet semilattices}\label{sec:meetsemilattice}
Let $(P,\preceq)$~be a nonempty poset equipped with a partial order relation $\preceq$. We denote the greatest lower bound of $x,y\in P$~by
\[
x\wedge y=\sup\{z\in P\mid z\preceq x\text{ and }z\preceq y\},
\]
which is called the {\it meet} of $x$ and $y$ provided that it exists. The triplet $(P,\preceq,\wedge)$~is called a {\it meet semilattice} if $x\wedge y$ exists for all $x,y\in P$. The poset $(P,\preceq)$ is called {\it locally finite}  if the interval
\[
[x,y]_P=\{z\in P\mid x\preceq z\preceq y\}
\]
is finite for all $x,y\in P,$ i.e., any two elements $x,y\in P$~are separated by at most a finite number of elements subject to the partial ordering $\preceq$.

A finite nonempty set $S\subset P$ is called {\it meet closed} if $x\wedge y\in S$ for all $x,y\in S$. On the other hand, the set $S$ is called {\it lower closed} if having any $x\in P$ with $x\preceq y$ for some $y\in S$ implies that $x\in S$. While a lower closed set is naturally meet closed, the converse is generally not true.

\subsection{Table of notations}\label{sec:notations}
The special notations used throughout this paper are listed in the following table.
\begin{table}[!h]
\begin{center}
\begin{tabular}{ll}
$\mathbf{\hat 0}_P$&The least element $\mathbf{\hat 0}_P$ of the poset $(P,\preceq_P)$~such that $\mathbf{\hat 0}_P\preceq_P x$ for all $x\in P$.\\
$*_P$&The $P$-convolution of functions $f,g\!:P\times P\to\mathbb{R}$ defined by setting\\
&\hfil$(f*_Pg)(x,y)=\displaystyle\sum_{x\preceq_P z\preceq_P y}f(x,z)g(z,y),\quad x,y\in P.$\hfil\\
$\zeta_P$&The incidence function $\zeta_P(x,y)=1$ if $x\preceq_P y$, $x,y\in P$, and $0$~otherwise.\\
$\delta_P$&The incidence function $\delta_P(x,y)=1$~if $x=y$, $x,y\in P$, and $0$~otherwise.\\
$\mu_P$&The M\"{o}bius function of the poset $(P,\preceq_P)$ is the inverse of $\zeta_P$~under $*_P$.\\
$|$&The divisibility relation of positive integers: $x|y$~$\Leftrightarrow$~$(y/x)\in\mathbb{Z}_+$,\ $x,y\in\mathbb{Z}_+$.
\end{tabular}
\end{center}
\end{table}

For an introduction to lattices and incidence functions, see for example~\cite{Mc,Stanley}.

\begin{remark}
The convolution of incidence functions as well as the Dirichlet convolution of one and several variables are all usually denoted as $f*g$ in the literature. In this paper, we adopt a different convention in order to distinguish these binary operations.
\end{remark}

\section{Positive definite functions defined on semilattices}\label{sec:poset}

The properties of functions $f\!:P\to\mathbb{R}$~with poset domains can be neatly characterized by introducing the notion of meet matrices.

\begin{definition}
Let $(P,\preceq,\wedge)$~be a meet semilattice, $S=\{x_1,\ldots,x_n\}$ a finite nonempty subset of $P$, and $f\!:P\to\mathbb{R}$ a function. The matrix $A=(S)_f$ defined by setting
\[
A_{i,j}=f(x_i\wedge x_j),\quad i,j\in\{1,\ldots,n\},
\]
is called the {\it meet matrix} of $S$ with respect to $f$.
\end{definition}

We can now define a general notion of positive definiteness for functions defined on meet semilattices.

\begin{definition}\label{def:meetpd}
Let $(P,\preceq,\wedge)$~be a meet semilattice. A function $f\!:P\to\mathbb{R}$ is called {\it positive definite} if the meet matrix $(S)_f$ is positive semidefinite for all finite sets $S\subset P$, $S\neq\varnothing$.\end{definition}

We can employ the properties of meet matrices to obtain a characterization for Definition~\ref{def:meetpd}.

\begin{theorem}\label{thm:meetchara}
Let $(P,\preceq,\wedge)$ be a locally finite meet semilattice. Let the finite nonempty sets $S_i\subset P$, $i\in \mathbb{Z}_+$, be a covering of $P$ such that
\[
P=\bigcup_{i=1}^\infty S_i.
\]
Then $f\!:P\to\mathbb{R}$~is positive definite if and only if $(S_m)_f$ is positive semidefinite for all $m\in\mathbb{Z}_+$.
\end{theorem}

\proof Without loss of generality, the covering $(S_i)_{i=1}^\infty$ can be assumed to be nested in the sense that $S_1\subset S_2\subset\cdots\subset P$ since it is always possible to construct a nested covering $(S_i')_{i=1}^\infty$ of $P$ by setting $S_1'=S_1$ and $S_i'=S_{i-1}'\cup S_i$, $i\geq 2$.

The ``only if'' direction follows immediately from the definition. To show the converse, let us assume that the matrix $(S_m)_f$ is positive semidefinite for all $m\in\mathbb{Z}_+$. Let $S$ be an arbitrary finite and nonempty subset of $P$. Then there is a positive integer $m$ such that $S\subset S_m$. The claim that the matrix $(S)_f$~is positive semidefinite now follows from the fact that it is a principal submatrix of the positive semidefinite matrix $(S_m)_f$, and every principal submatrix of a positive semidefinite matrix is always positive semidefinite (see, e.g., \cite[Observation~7.1.2]{HornJohnson}).\endproof

\begin{example}
If the locally finite meet semilattice $(P,\preceq,\wedge)$ consists of elements satisfying $x_1\preceq x_2\preceq\cdots\preceq x_n\preceq\cdots$, then it is called a {\it chain}. In this case, a covering for $P$~is given by the sets $S_m=\{x_1,\ldots,x_m\}$, and the positive definiteness of a function $f\!:P\to\mathbb{R}$ can be determined by proving the positive definiteness of the meet matrices $(\{x_1,\ldots,x_m\})_f$ for all $m\in\mathbb{Z}_+$.
\end{example}

The $LDL^{\rm T}$ decompositions of meet matrices provide an excellent way to characterize positive definite functions $f\!:P\to\mathbb{R}$. Since the $LDL^{\rm T}$ decomposition may be interpreted as an inertia preserving transformation of a matrix, we can deduce a criterion for the positive definiteness of functions $f\!:P\to\mathbb{R}$ using the decomposition theory of meet matrices.
\begin{theorem}\label{thm:meetcharacoro}
Let $(P,\preceq,\wedge,\mathbf{\hat{0}}_P)$ be a locally finite meet semilattice with the least element $\mathbf{\hat{0}}_P$. Let us assume that there exists a sequence of finite sets $S_i\subset P$, $i\in\mathbb{Z}_+$,~that cover $P$ such that
\[
P=\bigcup_{i=1}^\infty S_i.
\]
Then $f\!:P\to\mathbb{R}$~is positive definite if and only if
\[
(f_r*_P\mu_P)(\mathbf{\hat 0}_P,x)\geq 0\quad\text{for all }x\in P,
\]
where the restricted incidence function $f_r(\mathbf{\hat{0}}_P,x)=f(x)$, $x\in P$.
\end{theorem}
\proof The covering $(S_i)_{i=1}^\infty$~can be assumed to be nested in the sense that $S_1\subset S_2\subset\cdots\subset P$ (see the remark at the beginning of the proof of Theorem~\ref{thm:meetchara}).

We first remark that under the assumptions of this theorem, it is possible to construct another covering for $P$~consisting of only lower closed sets. We proceed by describing this construction.

Let us define the sets
\[
T_i=\{y\in P\mid y\preceq x,\ x\in S_i\},\quad i\in\mathbb{Z}_+.
\]
Let $i\in\mathbb{Z}_+$~be arbitrary. By the reflexivity of the partial order relation, it holds that $x\preceq x$ for all $x\in P$. Hence $S_i\subset T_i$. In consequence, the sets $T_i$, $i\in\mathbb{Z}_+$, form a covering for $P$. Moreover, the set $T_i$ must be finite since -- due to the assumption that the ambient meet semilattice is locally finite -- each $y\in T_i$ lies in the finite interval $y\in [\mathbf{\hat 0}_P,x]_P$ for some $x\in S_i$.

Next, let us show that the sets $T_i$ are lower closed. To this end, let $y\in T_i$ and $z\in P$~be such that $z\preceq y$. By construction, $y\preceq x$ for some $x\in S_i$. The transitive property of the partial order relation means that $z\preceq y\preceq x$ $\Rightarrow$ $z\preceq x$ $\Rightarrow$ $z\in T_i$. Hence $T_i$ is lower closed.

Due to the previous discussion, we may assume that the covering $(T_i)_{i=1}^\infty$ of $P$ consists of finite and nonempty lower closed sets $T_i$, $i\in\mathbb{Z}_+$. By Theorem~\ref{thm:meetchara}, it suffices to show that $(T_i)_f$ is positive semidefinite for all $i\in\mathbb{Z}_+$. Without loss of generality, we may assume that the elements of $T_i=\{x_1,\ldots,x_n\}$ are ordered $x_i\preceq x_j$ $\Rightarrow$~$i\leq j$. It follows from the decomposition theory of meet matrices~\cite[Theorem~12]{Bhat91} and from the formula of the M\"{o}bius function of a lower closed set~\cite[Example~1]{Hau96} that $(T_i)_f=EDE^{\rm T}$, where $E$ is an $n\times n$ matrix defined as $E_{i,j}=1$~if $x_j\preceq x_i$, $E_{i,j}=0$~otherwise, and $D={\rm diag}(d_1,\ldots,d_n)$, where
\[
d_i=(f_r*_P\mu_P)(\mathbf{\hat{0}}_P,x_i),\quad i\in\{1,\ldots,n\},
\]
where $f_r(\mathbf{\hat 0}_P,x)=f(x)$~for all $x\in P$.

The diagonal matrix $D$ is clearly positive semidefinite precisely when $d_i\geq 0$ for all $i\in\{1,\ldots,n\}$. Since $(T_i)_f$ is a congruence transformation of $D$, Sylvester's law of inertia implies that $(T_i)_f$ is positive semidefinite if and only if $d_i\geq 0$ for all $i\in\{1,\ldots,n\}$.\endproof

\begin{corollary}\label{cor:pcor}
Let $(P,\preceq,\wedge,\mathbf{\hat 0}_P)$ be a locally finite semilattice and let $S_1\subset S_2\subset\cdots\subset P$ be covering for $P$, where each $S_i$ is lower closed. Let $f\!:P\to\mathbb{R}$ be of the form $$f(x)=(g_r*_P\zeta_P)(\mathbf{\hat 0}_P,x),\quad  x\in P,$$ where $g_r(\mathbf{\hat 0}_P,x)=g(x)\geq 0$ for all $x\in P$. Then $f$ is positive definite.
\end{corollary}
\proof The M\"{o}bius inversion formula~\cite[Proposition~3.7.1]{Stanley} implies that
\[
(f_r*_P\mu_P)(\mathbf{\hat 0}_P,x)=g(x)\geq 0,
\]
where $f_r(\mathbf{\hat 0}_P,x)=f(x)$~for all $x\in P$.
\endproof

\subsection{Properties of positive definite functions of the form $f\!:P\to\mathbb{R}$}

Positive definiteness in the sense of Definition~\ref{def:meetpd} is preserved under the following fundamental arithmetical operations.
\begin{theorem}
Let $f,g\!:P\to\mathbb{R}$ be positive definite functions. Then
\begin{itemize}
\item[(i)] $af$~is positive definite for any scalar $a\geq 0$.
\item[(ii)] $f+g$ is positive definite.
\item[(iii)] $fg$ is positive definite.
\end{itemize}
\end{theorem}
\proof The proofs are carried out analogously to~\cite[Theorem~4.4]{HaukkanenMattila18}. However, for completeness, we give the proofs below.

Let $f\!:P\to\mathbb{R}$ and $g\!:P\to\mathbb{R}$ be positive definite. Let $S=\{x_1,\ldots,x_n\}$ be a finite, nonempty subset of $P$ ordered such that $x_i\preceq x_j$ $\Rightarrow$ $i\leq j$.

(i) Multiplication of $f$ by a constant $a\geq 0$~preserves the positive semidefiniteness of the respective meet matrices since $x^{\rm T}(S)_{af}x=ax^{\rm T}(S)_fx\geq 0$ for all $x\in\mathbb{R}^n$.

(ii) Positive definiteness is preserved under addition of functions $f$ and $g$ since $x^{\rm T}(S)_{f+g}x=x^{\rm T}(S)_fx+x^{\rm T}(S)_gx\geq 0$ for all $x\in\mathbb{R}^n$.

(iii) In the case $fg$, the corresponding meet matrix can be written as a Hadamard product $(S)_{fg}=(S)_f\circ (S)_g$ of two positive semidefinite matrices. By the Schur product theorem~\cite[Theorem~7.5.3]{HornJohnson}, it follows that the resulting matrix is also positive semidefinite.\endproof

Positive definite functions have the following monotonicity property.

\begin{theorem}\label{thm:monotonicity}
Let $f\!:P\to\mathbb{R}$ be positive definite. Then
\begin{itemize}
\item[(i)] $f(x)\geq 0$ for all $x\in P$.
\item[(ii)] $f(x)\leq f(y)$ for $x\preceq y$, $x,y\in P$.
\end{itemize}
\end{theorem}
\proof (i) Let $x\in P$ be arbitrary. Then by taking the singleton $S=\{x\}$, we obtain from the definition of positive definiteness that $(S)_f=f(x)\geq 0$.

(ii) Let $x,y\in P$ be such that $x\preceq y$. Then by considering the set $S=\{x,y\}$, we obtain
\[
0\leq \left|\begin{array}{cc}f(x\wedge x)&f(x\wedge y)\\ f(x\wedge y)&f(y\wedge y)\end{array}\right| =\left|\begin{array}{cc}f(x)&f(x)\\ f(x)&f(y)\end{array}\right|=f(x)f(y)-f(x)^2
\]
and the assertion follows immediately.\endproof

\section{Multivariate functions}\label{sec:multivariate}

As a special case of the general definition given in Section~\ref{sec:poset}, we study the positive definiteness of multivariate functions $f\!:P_1\times\cdots\times P_d\to\mathbb{R}$, where $(P_i,\preceq_{P_i},\wedge_{P_i})$ is a meet semilattice for all $i\in\{1,\ldots,d\}$. We begin by inspecting the poset $(P_1\times\cdots\times P_d,{\preceq}_{P_1\times\cdots\times P_d})$, where the product order ${\preceq}_{P_1\times\cdots\times P_d}$ is defined by setting
\[
\boldsymbol{x}\,{\preceq}_{P_1\times\cdots\times P_d}\, \boldsymbol{y}\quad\Leftrightarrow\quad x_i\preceq_{P_i} y_i\quad \text{for all }i\in\{1,\ldots,d\},
\]
where we denote $\boldsymbol{x}=(x_1,\ldots,x_d)\in P_1\times\cdots\times P_d$ and $\boldsymbol{y}=(y_1,\ldots,y_d)\in P_1\times\cdots\times P_d$ as ordered tuplets. 

Let us define the pairing $(\cdot,\cdot)_{P_1\times\cdots\times P_d}$ on the set $P_1\times\cdots\times P_d$ by setting
\[
(\boldsymbol{x},\boldsymbol{y})_{P_1\times\cdots\times P_d}:=(x_1\wedge_{P_1} y_1,x_2\wedge_{P_2} y_2,\ldots,x_d\wedge_{P_d} y_d)\quad\text{for all } \boldsymbol{x},\boldsymbol{y}\in P_1\times\cdots\times P_d.
\]
It is straightforward to verify that this pairing defines the meet of the poset $(P_1\times\cdots\times P_d,\preceq_{P_1\times\cdots\times P_d})$. Let $\boldsymbol{x}=(x_1,\ldots,x_d)\in P_1\times\cdots\times P_d$~and $\boldsymbol{y}=(y_1,\ldots,y_d)\in P_1\times\cdots\times P_d$. Then
\begin{align*}
&\sup\{\boldsymbol{z}\in P_1\times\cdots\times P_d\mid \boldsymbol{z}\preceq_{P_1\times\cdots\times P_d}\boldsymbol{x}\textup{ and }\boldsymbol{z}\preceq_{P_1\times\cdots\times P_d}\boldsymbol{y}\}\\
&=\sup\{(z_1,\ldots,z_d)\in P_1\times\cdots\times P_d\mid z_i\preceq_{P_i} x_i\text{ and }z_i\preceq_{P_i} y_i~\text{for all }i\in\{1,\ldots,d\}\}\\
&=(x_1\wedge_{P_1} y_1,\ldots,x_d\wedge_{P_d} y_d)=(\boldsymbol{x},\boldsymbol{y})_{P_1\times\cdots\times P_d}.
\end{align*}
This justifies identifying $(\boldsymbol{x},\boldsymbol{y})_{P_1\times\cdots\times P_d}=\boldsymbol{x}\wedge_{P_1\times\cdots\times P_d}\boldsymbol{y}$ for $\boldsymbol{x},\boldsymbol{y}\in P_1\times\cdots\times P_d$.

We begin by inspecting the special case $d=2$ in Subsection~\ref{sec:decomps} due to its superior notational simplicity. In Subsection~\ref{sec:multivariatedecomps} we will consider the $d$-variate setting.
\subsection{Decompositions of meet matrices of the form $(S\times T)_f$}\label{sec:decomps}
By Definition~\ref{def:meetpd}, the positive definiteness of functions $f\!:P\times Q\to\mathbb{R}$~can be established by considering meet matrices $A=(S\times T)_f$ with
\begin{align}
A_{i,j}=f(\boldsymbol{x}_i\wedge_{P\times Q} \boldsymbol{x}_j),\quad i,j\in\{1,\ldots,n\},\label{equ:meet2}
\end{align}
where $S\times T=\{\boldsymbol{x}_1,\ldots,\boldsymbol{x}_n\}$ is a subset of the poset $(P\times Q,\preceq_{P\times Q})$ ordered such that $\boldsymbol{x}_i\preceq_{P\times Q}\boldsymbol{x}_j$~$\Rightarrow$~$i\leq j$.

\begin{remark}\label{rem:multiindices}
If $S=\{x_1,\ldots,x_n\}$ and $T=\{y_1,\ldots,y_m\}$, then $$\boldsymbol{x}_i=(x_{1+\lfloor (i-1)/m\rfloor},y_{1+{\rm mod}(i-1,m)})$$ for $i\in\{1,\ldots,nm\}$. This is also the lexicographic ordering of the elements in $S\times T$. A connection between the Kronecker product and the lexicographic ordering will be used in Proposition~\ref{prop:decomp} to derive an $LDL^{\rm T}$ decomposition for matrices of the form~\eqref{equ:meet2}.
\end{remark}

We review some basic properties of meet matrices of the form $(S\times T)_f$. The following well known result applies to the M\"{o}bius function of a Cartesian product.
\begin{lemma}[cf.~{\cite[Proposition~3.8.2]{Stanley}}]
Let $(P,\preceq_P)$ and $(Q,\preceq_Q)$ be locally finite posets. The M\"{o}bius function of the poset $(P\times Q,\preceq_{P\times Q})$ is given by
\[
\mu_{P\times Q}(\boldsymbol{x},\boldsymbol{y})=\mu_P(x_1,y_1)\mu_Q(x_2,y_2),\quad \boldsymbol{x},\boldsymbol{y}\in P\times Q.
\]
\end{lemma}
In the following proposition we present the well-known factorization~\cite[Theorem~12]{Bhat91} in the case that the underlying poset is a Cartesian product of posets. This form supports our study of multivariate functions. For completeness, we  give a short proof below.
\begin{proposition}\label{prop:decomp}
Let $(P,\preceq_P,\wedge_P)$ and $(Q,\preceq_Q,\wedge_Q)$ be locally finite meet semilattices and suppose that $S=\{x_1,\ldots,x_n\}\subset P$ and $T=\{y_1,\ldots,y_m\}\subset Q$ are meet closed sets ordered such that $x_i\preceq_{P} x_j$~$\Rightarrow$~$i\leq j$ and $y_i\preceq_{Q} y_j$~$\Rightarrow$~$i\leq j$, respectively. Then 
\[
(S\times T)_f=(E\otimes F)\Lambda(E\otimes F)^\textup{T},
\]
where $\otimes$ denotes the Kronecker product, $E$~is the $n\times n$ matrix and $F$ is the $m\times m$ matrix defined by setting
\[
E_{i,j}=\begin{cases}1,&\text{if }x_j\preceq_P x_i\\ 0&\text{otherwise},\end{cases}\quad \text{and}\quad F_{i,j}=\begin{cases}1,&\text{if }y_j\preceq_Q y_i\\ 0&\text{otherwise},\end{cases}
\]
and $\Lambda={\rm diag}(c(1+\lfloor \frac{i-1}{m}\rfloor,1+{\rm mod}(i-1,m)))_{i=1}^{nm}$, where
\[
c(i,j)=\sum_{x_k\preceq_P x_i\text{ and }y_\ell\preceq_Q y_j}f(x_k,y_\ell)\mu_S(x_k,x_i)\mu_T(y_\ell,y_j).
\]
\end{proposition}
\proof Let us enumerate the rows and columns of the $nm\times nm$ Kronecker product $E\otimes F$ by the multi-indices $\boldsymbol{i},\boldsymbol{j}\in\{1,\ldots,n\}\times\{1,\ldots,m\}$ in lexicographic order, see Table~\ref{tab:kron} for an illustration. Using this convention, the elements of the Kronecker product can be expressed concisely as
\[
(E\otimes F)_{\boldsymbol{i},\boldsymbol{j}}=E_{i_1,j_1}F_{i_2,j_2},\quad \boldsymbol{i},\boldsymbol{j}\in\{1,\ldots,n\}\times\{1,\ldots,m\},
\]
where $\boldsymbol{i}=(i_1,i_2)$~and $\boldsymbol{j}=(j_1,j_2)$ are ordered pairs. By Remark~\ref{rem:multiindices}, we can number the corresponding elements of $S\times T$ by $\boldsymbol{x}_{\boldsymbol{k}}=(x_{k_1},y_{k_2})$ for $\boldsymbol{k}=(k_1,k_2)\in\{1,\ldots,n\}\times\{1,\ldots,m\}$. We also denote $\Lambda_{\boldsymbol{i}}=c(i_1,i_2)$ for $\boldsymbol{i}=(i_1,i_2)\in\{1,\ldots,n\}\times\{1,\ldots,m\}$.

Let us recall that by the M\"{o}bius inversion~\cite[Proposition~3.7.1]{Stanley} we now have
\begin{align}
\Lambda_{\boldsymbol{i}}=\sum_{\boldsymbol{k}: \boldsymbol{x}_{\boldsymbol{k}}\preceq_{P\times Q} \boldsymbol{x}_{\boldsymbol{i}}} f(\boldsymbol{x}_{\boldsymbol{k}})\mu_{S\times T}(\boldsymbol{x}_{\boldsymbol{k}},\boldsymbol{x}_{\boldsymbol{i}})\quad\Leftrightarrow\quad f(\boldsymbol{x}_{\boldsymbol{i}})=\sum_{\boldsymbol{k}:\boldsymbol{x}_{\boldsymbol{k}}\preceq_{P\times Q} \boldsymbol{x}_{\boldsymbol{i}}} \Lambda_{\boldsymbol{k}}\label{equ:mobius}.
\end{align}
Hence
\begin{align*}
&((E\otimes F)\Lambda(E\otimes F)^{\rm T})_{\boldsymbol{i},\boldsymbol{j}}\\
&=\sum_{\boldsymbol{k}\in\{1,\ldots,n\}\times\{1,\ldots,m\}}\Lambda_{\boldsymbol{k}}E_{i_1,k_1}F_{i_2,k_2}E_{j_1,k_1}F_{j_2,k_2}\\
&=\sum_{\boldsymbol{k}\in\{1,\ldots,n\}\times\{1,\ldots,m\}}\Lambda_{\boldsymbol{k}}\zeta_P(x_{k_1},x_{i_1})\zeta_Q(y_{k_2},y_{i_2})\zeta_P(x_{k_1},x_{j_1})\zeta_Q(y_{k_2},y_{j_2})\\
&=\sum_{\boldsymbol{k}\in\{1,\ldots,n\}\times\{1,\ldots,m\}}\Lambda_{\boldsymbol{k}}\zeta_P(x_{k_1},x_{i_1}\wedge_{P} x_{j_1})\zeta_Q(y_{k_2},y_{i_2}\wedge_{Q} y_{j_2}),
\intertext{where the last equality follows from the universal property $x\preceq y,z\Leftrightarrow x\preceq y\wedge z$. Now}
&((E\otimes F)\Lambda(E\otimes F)^{\rm T})_{\boldsymbol{i},\boldsymbol{j}}\\
&=\sum_{\boldsymbol{k}\in\{1,\ldots,n\}\times\{1,\ldots,m\}}\Lambda_{\boldsymbol{k}}\zeta_{P\times Q}((x_{k_1},y_{k_2}),(x_{i_1}\wedge_P x_{j_1},y_{i_2}\wedge_Q y_{j_2}))\\
&=\sum_{\boldsymbol{k}\in\{1,\ldots,n\}\times\{1,\ldots,m\}}\Lambda_{\boldsymbol{k}}\zeta_{P\times Q}(\boldsymbol{x}_{\boldsymbol{k}},\boldsymbol{x}_{\boldsymbol{i}}\wedge_{P\times Q}\boldsymbol{x}_{\boldsymbol{j}})\\
&=\sum_{\boldsymbol{k}:\boldsymbol{x}_{\boldsymbol{k}}\preceq_{P\times Q}\boldsymbol{x}_{\boldsymbol{i}}\wedge_{P\times Q}\boldsymbol{x}_{\boldsymbol{j}}}\Lambda_{\boldsymbol{k}}=f(\boldsymbol{x}_{\boldsymbol{i}}\wedge_{P\times Q}\boldsymbol{x}_{\boldsymbol{j}}),
\end{align*}
where the final equality follows from~\eqref{equ:mobius}.\endproof

\begin{table}[!t]
\begin{center}
$\begin{array}{cccccccc}
\substack{\text{column/row}\\ \text{index}}&(1,1)&(1,2)&\cdots&(1,m)&(2,1)&\cdots&(n,m)\\
\vspace{-.2cm}\\
(1,1)&E_{1,1}F_{1,1}&E_{1,1}F_{1,2}&\cdots&E_{1,1}F_{1,n}&E_{1,2}F_{1,1}&\cdots&E_{1,n}F_{1,m}\\
(1,2)&E_{1,1}F_{2,1}&E_{1,1}F_{2,2}&\cdots&E_{1,1}F_{2,n}&E_{1,2}F_{2,1}&\cdots&E_{1,n}F_{2,m}\\
\vdots&\vdots&\vdots&\ddots&\vdots&\vdots&\ddots&\vdots\\
(1,m)&E_{1,1}F_{m,1}&E_{1,1}F_{m,2}&\cdots&E_{1,1}F_{m,m}&E_{1,2}F_{m,1}&\cdots&E_{1,n}F_{m,m}\\
(2,1)&E_{2,1}F_{1,1}&E_{2,1}F_{1,2}&\cdots&E_{2,1}F_{1,n}&E_{2,2}F_{1,1}&\cdots&E_{2,n}F_{1,m}\\
\vdots&\vdots&\vdots&\ddots&\vdots&\vdots&\ddots&\vdots\\
(n,m)&E_{n,1}F_{m,1}&E_{n,1}F_{m,2}&\cdots&E_{n,1}F_{m,m}&E_{n,2}F_{m,1}&\cdots&E_{n,n}F_{m,m}
\end{array}$
\end{center}
\caption{Enumeration of the columns and rows of the Kronecker product $E\otimes F$ by the multi-indices $\boldsymbol{i},\boldsymbol{j}\in \{1,\ldots,n\}\times\{1,\ldots,m\}$ in lexicographic order.}\label{tab:kron}
\end{table}

\begin{remark}[cf.~{\cite[Proposition~2.4]{KorkeePhd}}]
If $S$ and $T$ are lower closed, then $\mu_{S\times T}$ coincides with the M\"{o}bius function $\mu_{P\times Q}$ of $(P\times Q,\preceq_{P\times Q})$. Otherwise it holds for meet closed $S$ and $T$ that
\[
\mu_{S\times T}(\boldsymbol{x}_i,\boldsymbol{x}_j)=\sum_{\substack{\boldsymbol{z}\preceq_{P\times Q} \boldsymbol{x}_j\\ \boldsymbol{z}\not\preceq \boldsymbol{x}_k,\ k<j}}\mu_{P\times Q}(\boldsymbol{x}_i,\boldsymbol{z}).
\]
\end{remark}

\begin{corollary} If in addition to the assumptions of Proposition~\ref{prop:decomp} it holds that $f(x,y)=g(x)g(y)$, then
\[
(S\times S)_f=(E\otimes E)(\Lambda\otimes\Lambda)(E\otimes E)^{\rm T},
\]
where $\Lambda={\rm diag}(c_1,\ldots,c_n)$ with
\[
c_i=\sum_{x_k\preceq_P x_i} g(x_k)\mu_S(x_k,x_i)\quad\text{for }i\in\{1,\ldots,n\}.
\]
\end{corollary}

We end this section by giving a connection between factorable functions (cf.~\cite[pg.~304]{Mc})  and the notion of positive definiteness.

\begin{definition}\label{def:factorable}
A function $f\!:P\times P\to\mathbb{R}$ is called {\it factorable} if it can be written as $f(x\wedge y,z\wedge w)=g(x\wedge z)h(y\wedge w)$ for some functions $g\!:P\to\mathbb{R}$ and $h\!:P\to\mathbb{R}$.
\end{definition}
\begin{lemma} Let $f\!:P\times P\to\mathbb{R}$ be a factorable function given as in Definition~\ref{def:factorable}. Then $f$ is positive definite if $g$ and $h$ are positive definite.
\end{lemma}
\proof Let $S=\{x_1,\ldots,x_n\}\subset P$~and $T=\{y_1,\ldots,y_m\}\subset Q$. Now
\[
(S\times T)_f=(S)_g\otimes (T)_h,
\]
where $\otimes$~denotes the Kronecker product. It is well known that the Kronecker product of positive semidefinite matrices is positive semidefinite and the assertion follows. \endproof

\subsection{Decompositions of meet matrices of the form $(S_1\times\cdots\times S_d)_f$}\label{sec:multivariatedecomps} Recursive application of Proposition~\ref{prop:decomp} can be used to yield the following generalized matrix decomposition.

\begin{proposition}\label{prop:decomp2}
Let $(P_i,\preceq_{P_i},\wedge_{P_i})$ be locally finite meet semilattices and suppose that $S_i=\{x_1^{(i)},\ldots,x_{n_i}^{(i)}\}\subset P_i$ are finite meet closed sets ordered such that $x_j^{(i)}\preceq x_k^{(i)}$~$\Rightarrow$~$j\leq k$ for all $i\in\{1,\ldots,d\}$. Then 
\[
(S_1\times\cdots\times S_d)_f=(E^{(1)}\otimes \cdots\otimes E^{(d)})\Lambda(E^{(1)}\otimes\cdots\otimes E^{(d)})^\textup{T},
\]
where $\otimes$ denotes the Kronecker product. Here, the $n_i\times n_i$~matrices $E^{(i)}$~are defined by setting
\[
E_{j,k}^{(i)}=\begin{cases}1,&\text{if }x_k^{(i)}\preceq_{P_i} x_j^{(i)}\\ 0&\text{otherwise},\end{cases}
\]
for all $i\in\{1,\ldots,d\}$. In addition, $\Lambda$ is a diagonal matrix with the diagonal elements 
\[
\sum_{\boldsymbol{k}:\boldsymbol{x}_{\boldsymbol{k}}\preceq_{P_1\times\cdots\times P_d}\boldsymbol{x}_{\boldsymbol{i}}}f(\boldsymbol{x}_{\boldsymbol{k}})\mu_{S_1\times\cdots\times S_d}(\boldsymbol{x}_{\boldsymbol{k}},\boldsymbol{x}_{\boldsymbol{i}}),
\]
where we set $\boldsymbol{x}_{\boldsymbol{k}}=(x_{k_1}^{(1)},\ldots,x_{k_d}^{(d)})$ and the multi-indices $\boldsymbol{i}\in\{1,\ldots,n_1\}\times\cdots\times\{1,\ldots,n_d\}$ are enumerated according to the lexicographic order.\end{proposition}

Proposition~\ref{prop:decomp2} is another special instance of~\cite[Theorem~12]{Bhat91}.

\section{Multivariate arithmetic functions $f\!:\mathbb{Z}_+^d\to\mathbb{R}$}\label{sec:multivariatearithmetic}
Let us begin by defining an extended GCD operator $(\cdot,\cdot)_d\!:\mathbb{Z}_+^d\times\mathbb{Z}_+^d\to\mathbb{Z}_+^d$ by setting
\[
(\boldsymbol{x},\boldsymbol{y})_d:=(\gcd(x_1,y_1),\ldots,\gcd(x_d,y_d))
\]
for all $\boldsymbol{x}=(x_1,\ldots,x_d)\in\mathbb{Z}_+^d$ and $\boldsymbol{y}=(y_1,\ldots,y_d)\in\mathbb{Z}_+^d$.

The extended GCD operator can be used to give a definition for the positive definiteness of multivariate arithmetic functions as follows.
\begin{definition}
A function $f\!:\mathbb{Z}_+^d\to\mathbb{R}$ is positive definite if the matrix $$(f((\boldsymbol{x},\boldsymbol{y})_d))_{\boldsymbol{x}\in S,\boldsymbol{y}\in S}$$ is positive semidefinite for all finite $S\subset \mathbb{Z}_+^d$, $S\neq \varnothing$.
\end{definition}

This definition can be expressed in terms of generalized Smith matrices.

\begin{theorem}\label{thm:chara}
A function $f\!:\mathbb{Z}_+^d\to\mathbb{R}$ is positive definite if and only if the matrix $(\{1,\ldots,m\}^d)_f$ is positive semidefinite for all $m\in\mathbb{Z}_+$.
\end{theorem}
\begin{proof} The assertion follows immediately from Theorem~\ref{thm:meetchara} since the sets $S_m=\{1,\ldots,m\}^d$, $m\in\mathbb{Z}_+$, constitute a covering for $\mathbb{Z}_+^d$.
\end{proof}

In the following, the {\it Dirichlet convolution} of $f\!:\mathbb{Z}_+^d\to\mathbb{R}$ and $g\!:\mathbb{Z}_+^d\to\mathbb{R}$ is defined as
\[
(f*_dg)(i_1,\ldots,i_d)=\sum_{\substack{k_j|i_j\\ j\in\{1,\ldots,d\}}}f(k_1,\ldots,k_d)g\big(\frac{i_1}{k_1},\ldots,\frac{i_d}{k_d}\big).
\]
The identity under $*_d$ is
\[
\delta_d(i_1,\ldots,i_d)=\delta(i_1)\cdots\delta(i_d),
\]
where $\delta(1)=1$~and $\delta(k)=0$ otherwise.

Let $\mu$ denote the arithmetic M\"{o}bius function defined by setting
\[
\mu(n)=\begin{cases}
1,&\text{if $n=1$,}\\
(-1)^m,&\text{if $n$ is the product of $m$ distinct primes},\\
0&\text{otherwise}.
\end{cases}
\]
Now by defining $\mu_d$~ as
\[
\mu_d(i_1,\ldots,i_d)=\mu(i_1)\cdots\mu(i_d)
\]
and letting $\zeta_d$~be defined as $\zeta_d(i_1,\ldots,i_d)=1$ for all $i_1,\ldots,i_d\in\mathbb{Z}_+$, we have
\[
\mu_d*_d\zeta_d=\delta_d.
\]

\begin{theorem}\label{thm:chara2} Let $d\geq 1$.
{\begin{itemize}
\item[(i)]A function $f\!:\mathbb{Z}_+^d\to\mathbb{R}$ is positive definite if and only if
\[
(f*_d\mu_d)(i_1,\ldots,i_d)\geq 0\quad\text{for all }i_1,\ldots,i_d\in\mathbb{Z}_+.
\] 
\item[(ii)]Let $f(x_1,x_2,\ldots,x_d)=g_1(x_1)g_2(x_2)\cdots g_d(x_d)$. Then $f$ is positive definite if and only if there exists $\mathcal{I}\subset\{1,\ldots,d\}$, where $\# \mathcal{I}$ is even, such that
\begin{align*}
&(g_i*_1\mu)(j) \leq 0\quad\text{for all }i\in \mathcal{I}\text{ and }j\in\mathbb{Z}_+\\
&\text{and}\quad(g_i*_1\mu)(j) \geq 0\quad\text{for all }i\not\in \mathcal{I},\ j\in\mathbb{Z}_+.
\end{align*}
\end{itemize}}
\end{theorem}
\proof (i) Due to Theorem~\ref{thm:chara}, it suffices to consider the positive semidefiniteness of matrices $(\{1,\ldots,m\}^d)_f$ for all $m\in\mathbb{Z}_+$. On the other hand, Proposition~\ref{prop:decomp2} implies that the matrix $(\{1,\ldots,m\}^d)_f$ is a congruence transformation of $\Lambda={\rm diag}(c(i_1,\ldots,i_d))_{(i_1,\ldots,i_d)\in\{1,\ldots,n\}^d}$ (multi-indices enumerated in lexicographic order), where
\[
c(i_1,\ldots,i_d)=\sum_{\substack{k_j| i_j\\ j\in\{1,\ldots,d\}}}f(k_1,\ldots,k_d)\mu\bigg(\frac{i_1}{k_1}\bigg)\cdots\mu\bigg(\frac{i_d}{k_d}\bigg)=(f*_d\mu_d)(i_1,\ldots,i_d).
\]
Since the congruence transformation preserves the inertia of any matrix, this concludes the proof of part (i).

(ii) For $f(x_1,\ldots,x_d)=g_1(x_1)\cdots g_d(x_d)$, it holds that
\begin{align*}
&(f*_d\mu_d)(i_1,\ldots,i_d)\\
&=\sum_{\substack{k_j|i_j\\ j\in\{1,\ldots,d\}}}f(k_1,\ldots,k_d)\mu\bigg(\frac{i_1}{k_1}\bigg)\cdots\mu\bigg(\frac{i_d}{k_d}\bigg)\\
&=\bigg(\sum_{k|i_1}g_1(k)\mu\bigg(\frac{i_1}{k}\bigg)\bigg)\cdots\bigg(\sum_{k|i_d}g_d(k)\mu\bigg(\frac{i_d}{k}\bigg)\bigg)\\
&=(g_1*_1\mu)(i_1)\cdots(g_d*_1\mu)(i_d),
\end{align*}
which yields the assertion.
\endproof

\section{Examples of positive definite functions of several variables}\label{sec:examples}

We begin this section by illustrating the monotonicity property of Theorem~\ref{thm:monotonicity} for two lattices.

\begin{example}[Monotonicity]
$ $

(a) Let $(P,\wedge,\mathbf{\hat 0}_P)=(\mathbb{Z}_+^2,(\cdot,\cdot)_2,1)$ be the two-dimensional divisor lattice. In this case, $\zeta_P(\boldsymbol{x},\boldsymbol{y})=1$ whenever $x_1|y_1$ and $x_2|y_2$ for $\boldsymbol{x}=(x_1,x_2)\in\mathbb{Z}_+^2$ and $\boldsymbol{y}=(y_1,y_2)\in\mathbb{Z}_+^2$ and $0$ otherwise.

(b) Let $(P,\wedge,\mathbf{\hat 0}_P)=(\mathbb{Z}_+^2,\min_2,1)$ be the two-dimensional MIN lattice, where $$\text{min}_2((x_1,y_1),(x_2,y_2))=(\min(x_1,x_2),\min(y_1,y_2))$$
for $(x_1,x_2),(y_1,y_2)\in\mathbb{Z}_+^2.$ In this case, $\zeta_P(\boldsymbol{x},\boldsymbol{y})=1$ whenever $x_1\leq y_1$ and $x_2\leq y_2$ for $\boldsymbol{x}=(x_1,x_2)\in\mathbb{Z}_+^2$ and $\boldsymbol{y}=(y_1,y_2)\in\mathbb{Z}_+^2$  and $0$ otherwise.

It is not difficult to see that here one can obtain an absolutely increasing positive definite function $f\!:P\to\mathbb{R}$ with respect to the partial order relation by defining $f(x)=(g_r*_P\zeta_P)(\boldsymbol{\hat 0}_P,x)$ with $g(x)=g_r(\boldsymbol{\hat 0}_P,x)=1$ for all $x\in P$ (see Corollary~\ref{cor:pcor}). Illustrations of the relative magnitude of values that these positive definite functions take with respect to each lattice (a) and (b) are given in Figure~\ref{fig:monotonicity}.

\begin{figure}[!t]
\begin{center}
\includegraphics[width=.49\textwidth]{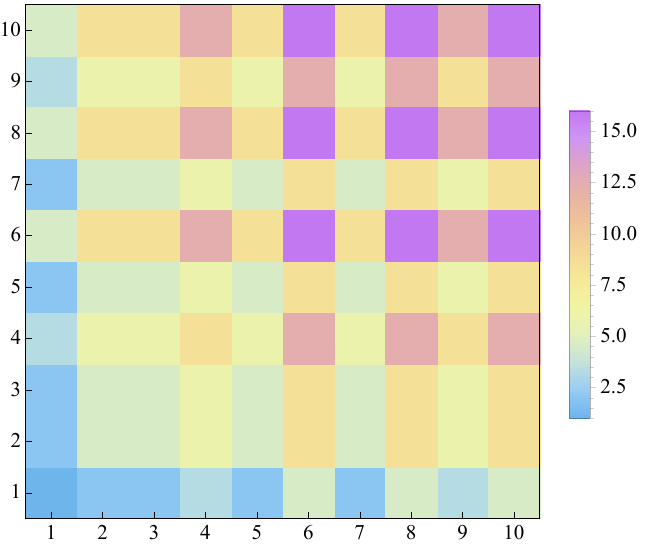}
\includegraphics[width=.49\textwidth]{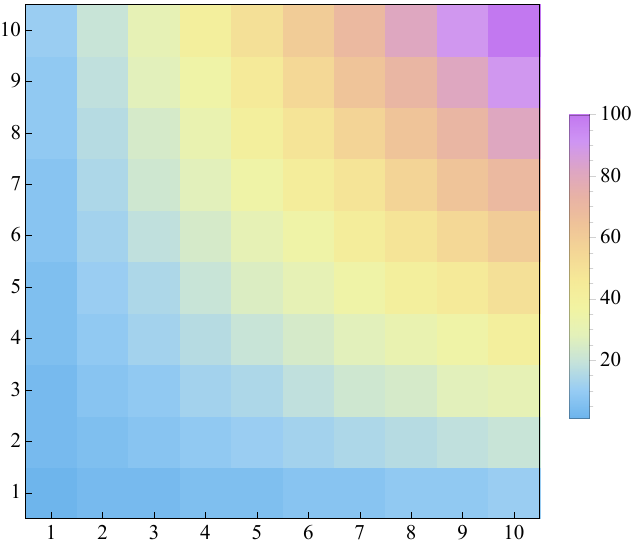}
\end{center}
\caption{Left: an example of the monotonicity of a positive definite function subject to the divisor lattice $(\mathbb{Z}_+^2,(\cdot,\cdot)_2)$ in the grid $\{1,\ldots,10\}^2$. Right: an example of the monotonicity of a positive definite function subject to the MIN lattice $(\mathbb{Z}_+^2,\min_2)$ in the grid $\{1,\ldots,10\}^2.$}\label{fig:monotonicity}
\end{figure}
\end{example}

\begin{example}\label{ex:gcdlcm} $ $

(a)~Let $f(x,y)={\rm gcd}(x,y)$. Now the generalized GCD matrix $A=(\{1,\ldots,m\}^2)_f$ has the form
\[
A_{(i_1,i_2),(j_1,j_2)}={\rm gcd}({\rm gcd}(i_1,j_1),{\rm gcd}(i_2,j_2))={\rm gcd}(i_1,i_2,j_1,j_2),
\]
where the rows and columns are enumerated by the multi-indices $(i_1,i_2),(j_1,j_2)\in\{1,\ldots,m\}^2$. Since $\{1,\ldots,m\}$ is meet closed, the matrix $A$ has rank $m$ and its full-rank submatrix can be found by considering the rows and columns with indices $(i_1,i_2),(j_1,j_2)\in\{(1,1),(2,2),\ldots,(m,m)\}$. In particular, $A$~admits to a congruence transformation
\[
PAP^{\rm T}=\begin{pmatrix}B&O\\ O&O\end{pmatrix},
\]
where $B_{i,j}={\rm gcd}(i,j)$, $i,j\in\{1,\ldots,m\}$, and $P$ is an elimination matrix (including pivots) targeting all linearly dependent rows of $A$. Since $B$~is well known to be positive definite for all $m\in\mathbb{Z}_+$, we conclude that the matrix $A$~is positive semidefinite and the function $f$~is positive definite in consequence. 

(b)~Let $f(x,y)={\rm lcm}(x,y)$. The characteristic polynomial of $(\{1,2\}^2)_f$ is
\[
p(\lambda)=(\lambda-1)(\lambda^3-6\lambda^2+1),
\]
which has one negative root. Therefore $f$~is not positive definite.
\end{example}

The trick we used in part (a) of Example~\ref{ex:gcdlcm} can be generalized.

\begin{example} Let $(P,\preceq,\wedge,\mathbf{\hat 0}_P)$ be a locally finite meet semilattice and suppose that there exists a covering $(S_i)_{i=1}^\infty$ consisting of finite and nonempty subsets of $P$. Let us assume that $g\!:P\to\mathbb{R}$ is a positive definite function.

Define the function $f\!:P\times\cdots\times P\to\mathbb{R}$ by setting $f(x_1,\ldots,x_d)=g(x_1\wedge\cdots\wedge x_d)$ and let $S=\{x_1,\ldots,x_m\}$ be a finite lower closed subset of $P$ ordered such that $x_i\preceq x_j$ $\Rightarrow$ $i\leq j$.\footnotemark\footnotetext{Notice that by the remarks in the beginning of the proof of Theorem~\ref{thm:meetcharacoro}, we can now find a lower closed covering of $P$. Thus it is sufficient to prove the positive semidefiniteness of the meet matrix $(S)_f$ for all finite and nonempty lower closed sets $S\subset P$.} It suffices to show the positive semidefiniteness of the meet matrix
\[
A_{(i_1,\ldots,i_d),(j_1,\ldots,j_d)}=g(x_{i_1}\wedge x_{j_1}\wedge \ldots\wedge x_{i_d}\wedge x_{j_d}),
\]
where $(i_1,\ldots,i_d),(j_1,\ldots,j_d)\in\{1,\ldots,m\}^d$ enumerate the rows and columns, respectively. Define the $m\times m$ matrix $B$ by setting $B_{i,j}=g(x_i\wedge_P x_j)$ for all $i,j\in\{1,\ldots,m\}$. Since $S$ is factor closed, it is especially meet closed, and thus $A$ contains only $m$ linearly independent rows and columns. Hence there exists an elimination matrix $P$ (with pivots) such that
\[
PAP^{\rm T}=\begin{pmatrix}B&O\\ O&O\end{pmatrix}.
\]
Now the multivariate function $f$ is positive definite precisely when the univariate function $g$ is positive definite. Some examples of positive definite multivariate arithmetic functions $f\!:\mathbb{Z}_+^2\to\mathbb{R}$ are $f(x,y)={\rm gcd}(x,y)^\alpha$ for $\alpha>0$ and $f(x,y)={\rm lcm}(x,y)^{\alpha}$ for $\alpha<0$.

\end{example}

\begin{example}\label{ex:nonnegative}
Assume that $f\!:\mathbb{Z}_+^2\to\mathbb{R}$ is of the form
\[
f=g*_2\zeta_2,
\]
where $g$ is always nonnegative. Then by Corollary~\ref{cor:pcor}, $f$ is positive definite.
\end{example}

Typical examples of nonnegative functions covered by Example~\ref{ex:nonnegative} are combinatorial number-theoretic functions counting the number of integers satisfying certain conditions. For example, number of solutions of certain congruences, see~\cite{Mc}.

\begin{example} Let us consider Ramanujan's sum $C(m,n)$, see~\cite{Mc}. It is well known that
\[
C(m,n)=\sum_{d|{\rm gcd}(m,n)}d\mu\big(\frac{n}{d}\big).
\]
Let
\[
P(m,n)=\begin{cases}
n,&\text{if $m=n$},\\
0&\text{otherwise.}
\end{cases}
\]
Now
\begin{align*}
C(m,n)=\sum_{d|m\text{ and }e|n}P(d,e)\zeta\big(\frac{m}{d}\big)\mu\big(\frac{n}{e}\big)=(P*_2(\zeta\mu))(m,n)
\end{align*}
and thus
\begin{align*}
(C*_2\mu_2)(m,n)&=(P*_2(\zeta\mu)*_2\mu_2)(m,n)\\
&=(P*_2\delta(\mu*_1\mu))(m,n)\\
&=\sum_{d|{\rm gcd}(m,n)}d\delta\big(\frac{m}{d}\big)(\mu*_1\mu)\big(\frac{n}{d}\big)\\
&=\begin{cases}
0,&\text{if $m\!\!\not |n$},\\
m(\mu*_1\mu)\big(\frac{n}{m}\big),&\text{if $m|n$}.
\end{cases}
\end{align*}
For suitable $m$~and $n$, the term $(\mu*_1\mu)(n/m)$ can be positive, negative, or zero since
\begin{align*}
&(\mu*_1\mu)(1)=1,\\
&(\mu*_1\mu)(p)=-2,\\
&(\mu*_1\mu)(p^2)=1,\\
&(\mu*_1\mu)(p^k)=0\text{ for }k\geq 3,
\end{align*}
and $\mu*_1\mu$ is multiplicative. Therefore, $C(m,n)$ is not positive definite.
\end{example}

\section*{Conclusions}

In this work, we have generalized the recently introduced notion of positive definiteness of arithmetic functions $f\!:\mathbb{Z}_+\to\mathbb{R}$ to functions defined on posets. In particular, we have shown that the major results proved for positive definite arithmetic functions in~\cite{HaukkanenMattila18} generalize in a natural way to functions $f\!:P\to\mathbb R$ with poset domains.


\bibliography{hikm}

\end{document}